\documentclass{amsart}
\usepackage{amsmath}
\usepackage{amsfonts}
\usepackage{amsthm, upref}
\usepackage{graphicx}
\usepackage{color}

\newcommand{\comment}[1]{}
\newcommand{\rr}{\mathbb{R}}

\newcommand{\iprod}[1]{\left\langle #1\right\rangle}
\newcommand{\mcal}{\mathcal}

\newcommand{\vol}[1]{\mathrm{Vol}\left(#1\right)}
\newcommand{\srf}{\mcal{S}}

\newcommand{\norm}[1]{\left\lVert #1 \right\rVert}
\newcommand{\dx}{d}

\newcommand{\eq}{\begin{equation}}
\newcommand{\en}{\end{equation}}
\newcommand{\abs}[1]{\left\lvert #1 \right\vert}

\newcommand{\ltwo}{\mathbb{L}^2}
\newcommand{\e}{E}
\newcommand{\cmp}{\mathbb{C}}
\newcommand{\re}{\mathfrak{Re}}

\newcommand{\comb}[2]{\begin{pmatrix}#1 \\ #2 \end{pmatrix}}
\newcommand{\G}{\mathcal{G}}

\begin{document}

\theoremstyle{plain}
\newtheorem{thm}{Theorem}
\newtheorem{lemma}[thm]{Lemma}
\newtheorem{lem}[thm]{Lemma}
\newtheorem{prop}[thm]{Proposition}
\newtheorem{cor}[thm]{Corollary}

\theoremstyle{definition}
\newtheorem{defn}{Definition}
\newtheorem{asmp}{Assumption}
\newtheorem{notn}{Notation}
\newtheorem{prb}{Problem}

\theoremstyle{remark}
\newtheorem{rmk}{Remark}
\newtheorem{exm}{Example}
\newtheorem{clm}{Claim}

\theoremstyle{plain}
\newtheorem*{fundamental}{Fundamental Region}

\title[A combinatorial analysis of interacting diffusions]{A combinatorial analysis of interacting diffusions}
\subjclass[2000]{abc}

\author[Sourav Chatterjee]{Sourav Chatterjee$^*$}
\thanks{$^*$Research partially supported by NSF grant DMS-0707054 and a Sloan Research Fellowship}
\address{367 Evans Hall \# 3860\\
Univ. of California at Berkeley\\
Berkeley, CA 94720-3860}
\email{sourav@stat.berkeley.edu}

\author[Soumik Pal]{Soumik Pal$^\dagger$}
\address{C-547 Padelford Hall\\ University of Washington\\ Seattle, WA 98195}
\email{soumik@math.washington.edu}
\thanks{$^\dagger$Research partially supported by N.S.F. grant DMS-0306194 to the probability group at Cornell University.}

\date{\today}
\maketitle


\begin{abstract}
We consider a particular class of $n$-dimensional homogeneous diffusions all of which have an identity diffusion matrix and a drift function that is piecewise constant and scale invariant. Abstract stochastic calculus immediately gives us general results about existence and uniqueness in law and invariant probability distributions when they exist. These invariant distributions are probability measures on the $n$-dimensional space and can be extremely resistant to a more detailed understanding. To have a better analysis, we construct a polyhedra such that the inward normal at its surface is given by the drift function and show that the finer structures of the invariant probability measure is intertwined with the geometry of the polyhedra. We show that several natural interacting Brownian particle models can thus be analyzed by studying the combinatorial fan generated by the drift function, particularly when these are simplicial. This is the case when the polyhedra is a polytope that is invariant under a Coxeter group action, which leads to an explicit description of the invariant measures in terms of iid Exponential random variables. Another class of examples is furnished by interactions indexed by weighted graphs all of which generate simplicial polytopes with $n !$ faces. We show that the proportion of volume contained in each component simplex corresponds to a probability distribution on the group of permutations, some of which have surprising connections with the classical urn models.  
\end{abstract}

\section{Introduction}

Consider the following two examples of stochastic processes. The first one, called the bang-bang process is classical and is particularly important in stochastic control theory. See the articles by Karatzas and Shreve \cite{karshbang}, Shreve \cite{shrevebang}, or the chapter by Warnecke \cite{warbang}. This is a one dimensional diffusion which solves the following SDE with a single real parameter $\alpha$:
\[
dX_t = -\alpha\;\text{sign}(X_t)dt + dW_t.
\]
It can be easily shown that the process is recurrent if and only if $\alpha$ is positive. In that case, the process has a unique reversible stationary distribution under which $\abs{X_t}$ is distributed as Exponential with rate $2\alpha$. 

The other example is the class of Brownian motions with rank-based interactions. This is a family of $n$ one dimensional diffusions which is parametrized by a single vector $\delta$ in $\rr^n$. These diffusions have an identity diffusion matrix and a drift that depends on the order in which the coordinates can be arranged in increasing values. If we think of each diffusion as recording the position of a particle moving on a line, then at any instant of time the particle with the $i$th smallest position gets an instantaneous drift $\delta_i$. The formal SDE for the diffusion can be described by
\eq\label{ranksde}
d X_t (i) = \sum_{j=1}^n \delta_j 1\left\{\;X_t(i)= X_{t}[j]\;\right\} dt + dW_t(i),   \quad i =1,2,\ldots,n,
\en
where $X_{t}[1] \le X_{t}[2] \le \ldots \le X_{t}[n]$ are the coordinates arranged in increasing order. The $W_t(i)$'s are asumed to be independent Brownian motions for some suitable underlying filtration. 

The rank-based interacting Brownian motions or closely related models have appeared in several veins of the literature. Extensive reviews can be found in the articles by Pal and Pitman \cite{palpitman} and Chatterjee and Pal \cite{chatpal}. Some of the recent work include the articles by Jourdain and Malrieu \cite{joumal} and Banner, Fernholz, and Karatzas \cite{atlasmodel}. Also see the related discrete time models by Ruzmaikina and Aizenman \cite{ruzaizenman}, and L-P. Arguin \cite{arguinPD}. 

The recurrence, transience, and other long term properties of the rank-based Brownian motions can be precisely determined. The following result is from \cite{palpitman}.

\begin{thm}[Theorem 4 in \cite{palpitman}]\label{theoremN} 
For $1 \le k \le n$ let
\eq
\label{alphak}
\alpha_k:= \sum_{i = 1}^k ( \delta_i -  \bar{\delta} ), \quad \bar{\delta}:=\frac{1}{n}\sum_{i=1}^n\delta_i.  
\en
For each fixed initial distribution of the $n$ particle system with
drifts $\delta_i$, the collection of laws of 
$X_t[n] - X_t{[1]}$ for $t \ge 0$ is tight if
and only if 
\eq
\label{conal}
\alpha _k >0  \mbox{ for all } 1 \le k \le N-1,
\en
in which case the following result holds:

The distribution of the spacings system $(X_t{[j+1]} - X_t{[j]},\;1\le j\le n-1)$ at time $t$ converges in total
variation norm as $t$ tends to infinity to a unique stationary distribution for
the spacings system, which is that of independent Exponential
variables $Y_k$ with rates $2\alpha_k$, $1 \le k \le n-1$.
Moreover, the spacings system is reversible at equilibrium.
\end{thm}

The independence of the spacings under the invariant distribution is somewhat puzzling since, due to the interaction, there is no independence between the spacing processes. The proof, which does not shed light on this phenomenon, invokes Williams's results on reflected Brownian motions \cite{williams87r}.

There are some similarities between the two examples. Both involve drift functions that are piecewise constant; in fact the drift is a single vector under the action of a group (sign flips for the former and permutations for the latter). Moreover, in both cases the invariant distribution involves independent Exponentials which provide a friendly description of an otherwise abstract probability measure.

We provide alternate proofs of these results as an application of the following general theory. Notice that the drift function in both the examples is the negative of the gradient (in the sense of distributions) of a positively homogenous (homogenous of degree one) function in $\rr^n$. For the bang-bang process, this function is $\alpha\abs{x}$, while for the rank-based processes it is given by $-\sum_{i=1}^n \delta_i x[i]$. Any continuous positively homogenous function which is nonnegative for all $x$ can be viewed as the Minkowski (or, the gauge) function of a certain body containing the origin. Since we consider piecewise constant drifts, these bodies are $n$ dimensional polyhedras. The corresponding stochastic process is shown to be recurrent when the polyhedra is a bounded polytopes, which in turn happens when $k$ is strictly positive for all $x\neq 0$. When this is the case, the invariant distribution for the diffusion can be obtained from the uniform distribution on the polytope. 

Let us now focus on the uniform distribution on an $n$-dimensional polytopes. If the polytope is a unit simplex $\{x: x_i \ge0,\; \sum_i x_i \le 1\}$, the uniform distribution can be effectively generated by dividing iid Exponential random variables by their total sum. Now suppose a polytope can be triangulated, i.e., decomposed as a union of $K$ simplices with disjoint interiors. It is a well-known theorem that such a triangulation is possible for every polytope. Every component simplex, say $S_i$, in the triangulation is a linear map $A_i$ of the unit simplex. Thus, an algorithm to generate a point uniformly from the polytope would be to generate a point $X$ uniformly from the unit simplex, choose $I$ between $\{ 1,2,\ldots,K\}$ with probability 
\eq\label{probi}
P\left( I=i \right)=\frac{\text{Vol}(S_i)}{\sum_{j=1}^K \text{Vol}(S_j)}, \quad i=1,2,\ldots,K,
\en
and let $Y=A_I X$. Such an $Y$ is clearly uniformly chosen from the polytope. 

A particularly explicit triangulation exists when the polytope is simplicial. That is, each of its extremal face is an $(n-1)$ dimensional simplex. One can then simply connect the origin (which is in the interior) with each of these faces to generate a nice triangulation. How does one check if a polytope is simplicial ? We demonstrate a simple condition when the symmetry group of the polytope is Coxeter, i.e., generated purely by reflections. Coxeter groups find applications in several ares of mathematics. The finite Coxeter groups include symmetries of regular polytopes and the Weyl groups of simple Lie algebras. They are usually defined formally as a set of generators and relations among them. However, we consider them in their original geometric form as treated by H. S. M. Coxeter in the classics \cite{cox1} and \cite{cox2}. The definitions and properties of irreducible group actions and Coxeter groups have been described in Subsection \ref{groupcase}. This is a particularly nice case, when not only the polytope is simplicial, but it is also regular, and hence the random variable $I$ is uniformly distributed among $\{1,2,\ldots,K\}$. To connect the dots with independent Exponentials, we simply need to describe the maps $A_i$'s. In fact, due to regularity, determining $A_1$ is enough, since the other maps are merely orbits under the group action. This is exactly the case for the bang-bang or the rank-based processes. The regularity corresponds to exchangeability among the particles, i.e., the condition that if the initial distribution is exchangeable among the coordinates, then so is the distribution at every other point of time. 

The second case we consider is not regular and does not involve any groups. The interaction is parametrized by all graphs with $n$ vertices and possible edge-weights. In this case, the maps $A_i$'s are simple and explicit. However, the probabilities in \eqref{probi} are not. In fact, these probabilities correspond to probabilities of various orderings of particles, increasingly arranged, under the invariant distribution.  Hence, these polytopes induce probability distributions on permutations of $n$ labels. We take up a few examples and show surprising connections with existing probability models on permutations. 
\medskip

In the regular case, one of the results we prove is the following. 

\begin{prop}\label{gpdiff1}
Consider the SDE 
\eq\label{sdegenform}
d X_t = b(X_t) dt + dW_t, \quad \text{where}\; W_t=(W_t(1), W_t(2),\ldots, W_t(n))
\en
is an $n$-dimensional Brownian motion. Assume that the drift function $b:\rr^n \rightarrow \rr^n$ is piecewise constant and satisfies
\eq\label{bcond}
b(\alpha x)=b(x), \quad \forall\;\alpha >0, \quad \text{and},\quad b(Ax)=Ab(x), \quad \forall\; A\in G,
\en
where $G$ is a finite irreducible group of orthogonal matrices.

Then the following conclusions hold.

\begin{enumerate}
\item Let $k(x)=-\iprod{x,b(x)}$. A sufficient condition for $X$ to be recurrent is that, for some non-zero vector $\lambda \in \rr^n$, we have
\[
k(x)=\max_{A\in G} \iprod{A\lambda, x}.
\]
In that case $X$ has a unique, reversible invariant distribution $\mu$ on $\rr^n$. The marginal law of $X_t$ converges in total variation to $\mu$ as $t$ tends to infinity.

\item If $G$ is Coxeter, there is a set of $n$ many linearly independent vectors $\{\eta_1, \eta_2, \ldots, \eta_n \}$ and $n$ many constants $\{\alpha_1, \alpha_2,\ldots,\alpha_n\}$ such that under $\mu$, the random variables
\[
Y_i= \alpha_i\iprod{A\eta_i, x}, \quad \text{if}\quad k(x)=\iprod{A\lambda,x}, \quad i=1,2,\ldots,n,
\]
are iid exponential random variables with rate two. 
\item  Additionally, if the stabilizer subgroup of $\lambda$ in $G$ is trivial, then the vectors $\{\eta_1, \eta_2, \ldots, \eta_n \}$ are determined as the generators of the conic hull of the set of vectors $\{\lambda - A\lambda, \; A\in G \}$. That is to say, every vector in the set $\{\lambda - A\lambda, \; A\in G \}$ can be represented as a linear combination of the subset $\{\eta_1, \eta_2, \ldots, \eta_n \}$ with nonnegative coefficients.

The constants $\alpha_1,\ldots,\alpha_n$  are the unique positive  coefficients of $\eta_1,\ldots,\eta_n$ in the expansion 
\begin{equation*}
\lambda= \sum_{i=1}^n \alpha_i \eta_i.
\end{equation*}  
\end{enumerate}
\end{prop}

Note that, we really do not need to know the details of the group structure to apply the previous result, except for the information that $G$ is Coxeter. As we show in the examples, necessary and sufficient conditions can be obtained if we have a better knowledge of the group structure. Finally, let us mention that a list of Coxeter groups up to isomorphisms is available and can be found in any standard textbook, say \cite{reflectiongroups}.

In Subsection \ref{examples}, we describe several families of interacting diffusions that can be analyzed by the previous theorem. They all appear as solutions to stochastic differential equations of the type \eqref{sdegenform} with a piecewise constant drift function satisfying conditions \eqref{bcond}, but involving different families of orthogonal groups. 

When the group is the group of permutation matrices, we get back rank based interactions. Using Proposition \ref{gpdiff1}, we provide an alternative proof (not involving reflected Brownian motions) of the Pal-Pitman result. 

The second class of examples are called \textit{sign-rank based interactions}. Here the drift vector not only changes when the coordinate values get permuted, but also, when when the signs of coordinates change. The relevant group is the one which generated by all the permutation matrices and all the diagonal matrices whose diagonal elements are either plus or minus one. In one dimension, this boils down to the simple Bang-bang process. 

The third class of examples are similar to sign-rank based processes, but with more constraints. Here, too, the drift vector changes when we permute coordinates. It also changes when we change signs of coordinates, but only when done in pairs. The group behind the curtain is generated by permutation matrices and diagonal matrices whose diagonal elements are $\pm 1$ with the additional constraint that only even number of $-1$'s are allowed.

Readers acquainted with the theory of Coxeter groups will recognize that the previous three examples correspond to the three well-known families of Coxeter groups, denoted by $\mcal{A}_n$, $\mcal{B}_n$, and $\mcal{D}_n$ for each $n\in \mathbb{N}$. In each case Proposition \ref{gpdiff1} allows us to formulate a simple sufficient condition for checking the existence of a unique invariant probability distribution and provides a complete description of the distribution in terms of independent Exponentials.   
\medskip

In the case of interaction through graphs we consider the following class of interacting diffusions. Let $\G$ be a graph on $n$ vertices where the vetices are labeled by $\{1,2,\ldots,n\}$. The edge between $i$ and $j$ have an associated edge weight $\beta_{ij}$, which is zero if there is no edge between the two vertices. Consider the SDE on $\rr^n$ given by
\[
dX_t(i) = \sum_{j=1}^n\beta_{ij}\text{sign}\left( X_t(j) - X_t(i) \right)dt + dW_t(i),\quad i=1,2,\ldots,n,
\]  
where, $W$ is again an $n$-dimensional Brownian motion. 
  
When all the edge weights are nonnegative, the model can be described by saying that the Brownian motions, which are indexed by the vertices of the graphs, get attracted towards one another. The constants $\beta_{ij}$ measure the strength of their attraction.  

Unless the graph is the complete graph with constant edge-weights, the interaction is not regular. However, if we define
\[
\bar{X_t} = \frac{1}{n}\sum_{i=1}^n X_t(i),
\]
the centered vector $(X(1)-\bar{X}, X(2) - \bar{X}, \ldots, X(n) - \bar{X})$ have an invariant distribution whenever the graph is connected and the edge-weights are nonnegative. 

For this class of interactions our main focus of investigation is the law of random permutation that takes indices of coordinates to their ranks under the invariant distribution. This is not uniform by virtue of not being regular. As an interesting example, we consider the case when we assume that each particle has a mass $m_i$, $i=1,2,\ldots,n$, and that $\beta_{ij}=m_im_j$ in the sense that the strength of the mutual attraction is proportional to the product of their masses. Due to this \textit{gravitational} intuition, under the invariant distribution, we should expect heavier particles to stay at the middle of the pile, while the lighter ones should be at the edge (the Sun being at the center, and Pluto at the far end).  

In general, this is very difficult to prove. However, in one particular case, this becomes apparent. For any $\alpha > 0$, consider $n$ particles with the interaction described in the previous paragraph, where the mass of the first particle is $\alpha$ and the rest of the masses are $1$. Let $\sigma(1)$ denote the rank of the first particle under the invariant distribution of the centered vector $(X(1)-\bar{X}, X(2) - \bar{X}, \ldots, X(n) - \bar{X})$. A surprising connection with Polya's urn scheme emerges. We prove the following.

\begin{prop}\label{massbeta1}  Consider a Polya's urn scheme which initially has $2\alpha$ red balls and $2\alpha$ black balls. At every step one picks ups a ball at random, returns the ball to the urn and adds an extra ball of the same color. Then, the distribution of $\sigma(1)-1$ is the same as the number of red balls picked when we run the urn scheme described above for $n-1$ steps. 

In particular, the sequence of random variables $\sigma(1)/n$ converges weakly to the Beta$(2\alpha,2\alpha)$ distribution as $n$ tends to infinity.
\end{prop}

\subsection{Outline of the paper}
In the next section we describe the set-up of the paper and prove general results about recurrence of interacting diffusions and their invariant distributions when they exist. Section \ref{cones} describes the combinatorics involved in the invariant distributions. Subsection \ref{groupcase} describes the connection with Coxeter groups followed by several examples in Subsection \ref{examples}. The following subsection \ref{graphcase} proves results about interactions parametrized by graphs.

\section{Diffusions with piecewise constant drift}

Consider a sequence of $n$-dimensional cones $C_1, C_2, \ldots, C_r$ whose interiors are disjoint and the closure of their union is the entire space $\rr^n$. Ignoring their mutual intersections (a set of measure zero) they can be thought of as a partition of $\rr^n$. Let $b:\rr^n \rightarrow \rr^n$ be a function that is constant over each $C_i$. In this section we prove some general results about the class of diffusions which satisfy the following stochastic differential equation:
\eq\label{sde1}
d X_t = b(X_t) dt + dW_t, 
\en
where $W_t=(W_t(1), W_t(2),\ldots, W_t(n))$ is an $n$-dimensional Brownian motion. 

The existence and uniqueness in law is immediate by an application of Girsanov's theorem. Define the function $k(x)= -\iprod{x,b(x)}$. Since $b$ is constant over cones, it follows that $k$ is a positively homogeneous function. For example, $k$ could be the Minkowski functional (i.e., the gauge function) of a convex body containing the origin. If $k$ is continuous, by virtue of being piecewise linear, it follows easily that $b$ is the negative of the gradient of $k$ in the sense of distributions. In that case the SDE in \eqref{sde1} is an example of the Langevin equation. The following proposition is well-known about Langevin SDE.

\begin{prop}\label{invariantdist}
Consider the stochastic differential equation \eqref{sde1}. Let $k:\rr^n \rightarrow \rr$ be a continuous function such that $b$ represents $-\nabla k$ in the sense of distribution.
Assume that $\exp(-2k(x))$ is integrable. Then the probability distribution given by the un-normalized density $\exp(-2k(x))\dx x$ provides a reversible, invariant probability distribution $\mu$ for the process in \eqref{sde1}.
\end{prop}

\begin{proof} Let $\mu$ be the measure on $(\rr^n, \mcal{B}(\rr^n)$ defined by $\mu(\dx x)= e^{-2k(x)}\dx x$. Consider the Sobolev space, $H^{1,2}$, of all measurable functions $f:\rr^n \rightarrow \rr$, such that $f$ and all of its partial derivatives $\partial f/\partial x_i$, $i=1,2,\ldots,n$ (in the sense of distributions) are in $\ltwo(\mu)$. Then we can define the following symmetric bilinear form, on the domain $H^{1,2}$, given by
\[
\mcal{E}(f,g)= \int_{\rr^n} \langle \nabla f,\nabla g \rangle e^{-2k(x)}\dx x. 
\]
Since $e^{-2k(x)}$ is never zero, it follows that $H^{1,2}$ is a Hilbert space. Thus, $\mcal{E}$ is closed, since it is defined everywhere on the Hilbert space $H^{1,2}$. It is also known to be $\mcal{E}$ is Markovian (see, e.g., \cite{dirichlet}, example 1.2.1.). Thus, it is clear that this is a Dirichlet form in $\ltwo(\rr^n,\mu)$.

By Theorem 1.3.1 in \cite{dirichlet}, we claim the existence of a unique non-positive definite self-adjoint operator $\mcal{L}'$ on $H^{1,2}$ such that
\[
\mcal{E}(f,g) = \langle \sqrt{-\mcal{L}'}f, \sqrt{-\mcal{L}'}g \rangle_{\mu}, \quad \forall \; f,g \in \; H^{1,2}.
\]
Here $\iprod{\cdot}_{\mu}$ refers to the usual inner product in $\ltwo(\mu)$. Or, in other words, (Corollary 1.3.1 of \cite{dirichlet}) there is a unique self-adjoint operator $\mcal{L}'$ on a domain $\mcal{D}(\mcal{L}')\subseteq H^{1,2}$ such that
\begin{equation}\label{uniqueop}
\mcal{E}(f,g) = \iprod{-\mcal{L}'f, g}_{\mu},\quad \forall\; f \in \mcal{D}(\mcal{L}'),\;\forall\; g \in H^{1,2}.
\end{equation}

We now show that \eqref{uniqueop} is satisfied by a multiple of the generator of the Markov process in \eqref{sde1}. The generator, $\mcal{L}$, is given by
\[
\mcal{L}f = \langle b, \nabla f \rangle + \frac{1}{2}\Delta f.
\]
By our assumption $b=-\nabla k$ takes finitely many values. Thus, we can define $\mcal{L}$ on the domain
\[
H^{2,2} = \left\{f \in \ltwo(\mu)\;\Big\vert\; \frac{\partial f}{\partial x_i}\in \ltwo(\mu), \text{and}\; \frac{\partial^2 f}{\partial x_i^2} \in \ltwo(\mu)\;\forall\; i=1,2\ldots,n.\;  \right\}.
\]
It is clear that the domain of $\mcal{L}$ above is a subset of $H^{1,2}$.

We claim that $2\mcal{L}$ satisfies \eqref{uniqueop}. In that direction, consider any $f\in H^{2,2}$ and any $g\in H^{1,2}$, we have
\begin{equation}\label{equalgen}
\begin{split}
\int_{\rr^n}\mcal{L}(f)ge^{-2k(x)}\dx x &= \int_{\rr^n} \left( \langle b, \nabla f \rangle + \frac{1}{2}\Delta f \right)g e^{-2k(x)}\dx x\\
&= \int_{\rr^n} \langle b, \nabla f \rangle g e^{-2k(x)}\dx x + \int_{\rr^n}\frac{1}{2}\Delta f g e^{-2k(x)}\dx x\\
&= -\int_{\rr^n} \langle \nabla k, \nabla f \rangle g e^{-2k(x)}\dx x - \frac{1}{2}\int_{\rr^n}\langle \nabla f , \nabla(g e^{-2k(x)}) \rangle\dx x \\
&= -\frac{1}{2}\int_{\rr^n}\langle \nabla f, \nabla g\rangle e^{-2k(x)}\dx x =-\frac{1}{2}\mcal{E}(f,g).
\end{split}
\end{equation}
Note that the boundary terms are zero in the integration by parts above since both $\partial f/\partial x_i$ and $g$ are in $\ltwo(\mu)$, and thus
\[
\frac{\partial f}{\partial x_i}g e^{-2k(x)}\Big\vert^{\infty}_{-\infty}=0, \quad \forall i=1,2,\ldots,n.
\]
We can rewrite \eqref{equalgen} as
\[
\mcal{E}(f,g) = \iprod{-2\mcal{L}f,g},\quad \forall\; f\in H^{2,2},\; \forall\; g \in H^{1,2}
\]
which, compared with \eqref{uniqueop}, proves that $2\mcal{L}$ to be the unique operator associated with the Dirichlet form $\mcal{E}$. Further, from self-djointness of $\mcal{L}$, we infer
\begin{equation}\label{reverse}
\langle \mcal{L}f, g \rangle_{\mu} = \iprod{f,\mcal{L}g}_{\mu}, \quad \forall\; f,g \in H^{2,2},
\end{equation}
where $\iprod{\cdot,\cdot}_{\mu}$ refers to the usual $\ltwo$ inner product. We can take $g\equiv 1$ to get that $\mu$ is an invariant measure for the process $X_t$. This proves the claim.
\end{proof}

When is the function $\exp(-2k(x))$ integrable ? This question is critical to both the recurrence property of the diffusion process as well as the existence of a unique long term stationary distribution. Its answer, however, is geometric in nature. It is intuitive that $k$ needs to be nonnegative. Notice that if $k$ is nonnegative, by virtue of being positively homogeneous, it is the gauge function (Minkowski functional) of a set containing the origin. That is to say, if we define the \textit{unit ball} and the \textit{surface} given by $k$ respectively as
\begin{equation}\label{ballandsurface}
\begin{split}
C&=\{x\;:\; k(x) \le 1\},\quad\srf=\{ x\;: \; k(x)=1\},
\end{split}
\end{equation}
it is not difficult to see that $k$ satisfies the relation $k(x)=\inf\left\{ \alpha > 0:\; x\in \alpha C \right\}$. We have the following definition.

\begin{defn}\label{perfect}
A continuous, nonnegative, positively homogeneous function $k:\rr^n \rightarrow \rr\cup\{\infty\}$ is said to be \emph{irreducible} if it satisfies $k(x)=0$ if and only if $x=0$.
\end{defn}

\begin{lemma}\label{firstproperties}
If $k$ is continuous and irreducible either $k(x) \ge 0$, $\forall\; x\in \rr^n$, or $k(x)\le 0$, $\forall\; x\in \rr^n$. Moreover, if $k(x) > 0$, $\forall\; x \in \rr^n, x\neq 0$, then $\vol{C} < \infty$.
\end{lemma}

\begin{proof}
Suppose that there are points $x_0$ and $x_1$ such that $k(x_0) > 0$ and $k(x_1) < 0$. We can choose a continuous curve $\gamma_t,\; t \in [0,1]$, in $\rr^n$ such that $\gamma_0=x_1$ and $\gamma_1=x_2$ and $0 \notin \gamma[0,1]$. Since $k$ is continuous, by the intermediate value theorem, there exists a $t^* \in (0,1)$ such that $k(\gamma_{t^*}) = 0$ but $\gamma_{t^*}\neq 0$. But this is impossible if $k$ is irreducible, and we have proved the first assertion of the lemma.

For the second assertion, we need to show that $C$ is bounded. Suppose, on the contrary, we can find a sequence $\{x_n\} \subseteq C$ such that $\lim_{n\rightarrow \infty}\norm{x_n}=\infty$. One can assume that $\norm{x_n} \ge 1$, for all $n \in \mathbb{N}$. Then the points $y_n = x_n / \norm{x_n}$ satisfy
\[
k(y_n) = k(x_n) / \norm{x_n} \le 1,\quad \forall\; n \in \mathbb{N}.
\] 
Thus, $y_n \in C$, for all $n=1,2,\ldots$. However, there exists a subsequence of $y_n$, say $\{y_{n_m}\}$ such that $\lim_{m\rightarrow \infty} y_{n_m}=z$, for some $z$ with $\norm{z}=1$. Hence, by continuity of $k$, we infer
\[
k(z) = \lim_{m \rightarrow \infty} k(y_{n_m}) = \lim_{m \rightarrow \infty} \frac{k(x_{n_m})}{\norm{x_{n_m}}} = 0.
\]
The final equality is due to the fact that $0 \le k(x_n) \le 1$ for all $n$ and $\lim \norm{x_n} =\infty$. Since $z\neq 0$, this contradicts our assumption that $k$ is irreducible. Hence we are done. 
\end{proof}

We now show that the process in \eqref{sde1} is Harris recurrent if $k$ is nonnegative and irreducible. It then follows (see \cite{durrettstoccalculus}, Section 7.5) that it has a unique invariant measure $\mu$ described above in Proposition \ref{invariantdist}. Moreover, if $P_t(x)$ is the marginal distribution of $X_t$ when $X_0=x$, then
$\lim_{t\rightarrow \infty}\norm{P_t(x) - \mu}_{\text{TV}} =0$. Here $\norm{\cdot}_{\text{TV}}$ refers to the total variation norm on measures. 

The following claim settles the argument.

\begin{lem}\label{homogrecurrent} Consider the notations and assumptions in Proposition \ref{invariantdist}. Suppose that the function $k$ is a nonnegative, irreducible, positively homogeneous function. Then the process $X_t$ is recurrent. 
\end{lem}

\begin{proof}
We will use Corollary 7.5.4 in \cite{durrettstoccalculus}. We need to consider the quantity
\[
d(x) = n + 2\iprod{x,b(x)},\quad x \in \rr^n.
\]
By our definition we have $\iprod{x,b(x)}=-k(x)$. Thus, $d(x) = n - 2k(x)$.

Now, since $k$ is non-negative and irreducible, it is growing to infinity \emph{uniformly} in all directions radially outward from zero. The way to see this is to note
\[
k(x) = \norm{x}k(x/\norm{x})\;\ge\; \norm{x} \inf_{\norm{y}=1}k(y)= c_1 \norm{x}
\]
The constant $c_1=\inf_{\norm{y}=1}k(y)$ is positive since $k$ is a strictly positive continuous function on the compact set $\{ y:\norm{y}=1\}$.

Now, if we fix an $\epsilon > 0$, there exists $R > 0$ such that $d(x) < -\epsilon$ for all $x$ with $\norm{x} > R$. Let $T_R$ be the first hitting time of the compact set $B_R=\{x\in \rr^n:\; \norm{x}\le R\}$. By Corollary 7.5.4 in \cite{durrettstoccalculus}, we immediately obtain $\e^{x}(T_R) \le{\norm{x}^2}/{\epsilon}$. Thus $B_R$ gets visited infinitely often and hence the process $X_t$ is recurrent.
\end{proof}

The integrability of the function $\exp(-2k(x))$ requires precisely the same condition as in the last lemma.

\begin{lemma}\label{homisinteg}
Suppose that $k:\rr^n \rightarrow \rr$ is a nonnegative, irreducible, continuous, positively homogeneous function. Then the $\exp(-2k(x))$ is an integrable function.
\end{lemma}

To prove the previous lemma we need the following \textit{polar decomposition} formula. For any nonnegative, irreducible $k$, define the surface projection function $\Theta:\rr^n \rightarrow \srf$ by
\begin{equation}\label{whatistheta}
\Theta(x) = \frac{x}{k(x)},\quad x \in \rr^n \backslash \{0\}, \;\text{and}\; \Theta(0) = 0.
\end{equation}
Clearly the range of $\Theta$ is the surface $\srf$ defined in \eqref{ballandsurface}.

We take the following two measurable spaces
\[
M_1 = \Big(\rr^n, \mcal{B}(\rr^n)\Big)\quad\text{and}\quad M_2 = \Big( \rr^+ \times \srf,\;\; \mcal{B}(\rr^+)\otimes \mcal{B}(\srf)\Big). 
\]
One can construct a measurable map $T:M_1 \rightarrow M_2$ given by $T(x) = \left( k(x), \Theta(x)\right)$, $\forall\; x\in \rr^n$. It clearly follows from the definition that $T$ is a one-to-one map. We prove the following slightly general result for the sake of completeness.

\begin{lem}\label{changeofvariable} For any nonnegative, irreducible, positively homogeneous function $k:\rr^n \rightarrow \rr^+ \cup \{\infty\}$, and any integrable $f:\rr^n \rightarrow \rr$, we have
\begin{equation}\label{cngvble}
\int_{\{x:k(x) < \infty\}} f(x)\dx x = n \vol{C} \int_{0}^{\infty} r^{n-1} \int_{\srf}f(r\cdot z)\dx \mu(z)\; \dx r,
\end{equation}
where $dx$ refers to the $n$-dimensional Lebesgue measure and $\mu$ is the \textit{cone measure} on $\srf$ defined as
\begin{equation}\label{conemsr}
\mu(E) = \frac{1}{\vol{C}}\vol{T^{-1}([0,1]\times E)}, \quad \forall \; E \in \mathcal{B}(\srf).
\end{equation}
\end{lem}

\begin{proof}
We first prove \eqref{cngvble} for functions $f$ equal to indicators of sets $A= T^{-1}\left([0,b]\times E\right)$, where $E \in \mathcal{B}(\srf)$ and $b \ge 0$. We have
\begin{eqnarray*}
\vol{A} = \int 1_{A}\dx x &=& \int 1{\{k(x) \le b\}}1{\{\Theta(x) \in E\}}\dx x.
\end{eqnarray*}
If we make let $y=b^{-1}x$, then, by positive homogeneity of $k$, one can write the last equation as
\[
\begin{split}
\vol{A} &= b^n \int 1{\{k(y) \le 1 \}}1{\{\Theta(y) \in E\}}\dx y = b^n \vol{T^{-1}([0,1]\times E)} \\ &= b^n \vol{C}\cdot \mu(E) = \vol{C} n \int_{0}^{b} r^{n-1} \int_E \dx \mu \; \dx r,
\end{split}
\]
which proves \eqref{cngvble} for this particular case. The rest of the argument follows from standard measure theoretic approximation results.
\end{proof}

\begin{proof}[Proof of Lemma \ref{homisinteg}] Since $2k$ is another nonnegative, irreducible, positively homogeneous function, it suffices to show that $\exp(-k(x))$ is integrable.

By Lemma \ref{firstproperties}, the set $C=\{x\in \rr^n:\; k(x)\le 1\}$ has a finite volume, and hence the cone measure on $\srf=\partial C$ is well-defined. From the change of variable formula in Lemma \ref{changeofvariable} we then obtain
\[
\begin{split}
\int_{\rr^n}e^{-k(x)}\dx x &= n\vol{C} \int_{0}^{\infty}r^{n-1}e^{-r}\dx r\\
&= {n}\vol{C} \int_{0}^{\infty}s^{n-1}e^{-s}\dx s= \vol{C} {n}\Gamma\left({n}\right)\\
&=\vol{C}\Gamma\left({n}+1\right) < \infty.
\end{split}
\]
\end{proof}
Note that, under the probability measure with the unnormalized density $e^{-k(x)}$, the random variable $k(X)$ is always distributed as Gamma($n$) irrespective of $k$ and independently of the vector $\Theta(X)$. Similarly, under the uniform measure on $C$, the random variable $k(X)$ is always distributed as Beta$(n,1)$ independently of $\Theta(X)$. Under both these measures, $\Theta(X)$ has the same law. This provides a link between the two probability measures which is important in their understanding.

\bigskip

What can be recovered when $k$ is not irreducible ? In general little, except when $k$ is the gauge function of a lower dimensional polytope. The following proposition generalizes Proposition \ref{invariantdist}. 
  
\begin{prop}\label{additivesplit}
Let $k:\rr^n \rightarrow \rr$ be a continuous function whose derivative in the sense of distribution is represented by a bounded function. Suppose there exists a subspace $H \subseteq \rr^n$ such that if $y \in H$ and $z \in H^{\perp}$, then
\[
k(y + z) = k_1(y) + k_2(z).
\]
Additionally, assume that $k_1$ is a nonnegative, irreducible, positive homogeneous function on $H$. Consider the solution to \eqref{sde1} when $b(x)$ represents $-\nabla{k(x)}$, where $\nabla$ is in the sense of distributions. Assume that a solution to \eqref{sde1} with this drift exists. Let $A:\rr^n \rightarrow H$ be a projection matrix onto the subspace $H$. Then 
\begin{enumerate}
\item the process $Y_t=AX_t$ has a unique reversible stationary probability distribution $\mu$.
\item Suppose $k'$ is any other function defined as
\[
k'(x) = k_1(Ax) + k_2'(x-Ax),\quad \forall\; x\in \rr^n,
\] 
for some non-negative function $k_2'$ such that $\exp(-2k')$ is integrable. Then $\mu$ is the law of the random vector $Y=AX$, where $X$ is a random vector with density proportional to $\exp(-2k')$. 
\end{enumerate}
\end{prop}

\begin{proof} Define $k'_2:\rr^n\rightarrow \rr$ to be
\begin{equation}\label{whatisk2prime}
k_2'(x) = \sup_{1\le j\le n-d}\vert \iprod{x,\delta_j}\vert,
\end{equation}
where $\delta_i,\;i=1,\ldots,n-d$ is an orthogonal basis of $H^{\perp}$. Modify the function $k$ by defining
\[
k'(x) = k_1(y) + k'_2(z),\quad \forall\; y\in H,\; z\in H^{\perp},\; x=y+z.
\]
\medskip

\noindent $\bullet$ We claim that $k'$ is a non-negative, irreducible, positively homogeneous function on $\rr^n$. 
Let us verify the condition in definition \ref{perfect}. Since $k_1$ and $k_2'$ are both positive, for $x=y+z$, $y\in H$, $z\in H^{\perp}$, we get
\[
\{h'(x)=0\}\; \Leftrightarrow\; \{k'(x)=0\}\;\Leftrightarrow\;\{ k_1(y)=0\}\; \text{and} \; \{k_2'(z)=0\}.
\]
But, by assumption for $k_1$ and by construction for $k_2'$, we get that $k'(x)=0$ if and only if $y=0$ and $z=0$, that is, if and only if $x=0$. This proves our claim.

Consider the solutions of the following two stochastic differential equations
\begin{eqnarray}
\dx X_t &=& -\nabla k(X'_t)\dx t + \dx W_t, \label{orgeqn}\\
\dx X'_t &=& -\nabla k'(X'_t)\dx t + \dx W_t, \label{alteqn}
\end{eqnarray}
where the drifts are function representing the derivatives in the sense of distributions. That $X_t'$ exists (in the weak sense) is clear.
\medskip

\noindent $\bullet$ We claim that $X'$ has a unique reversible invariant distribution $\mu'$ with unnormalized density $\exp(-2k'(x))\dx x$.  From Lemma \ref{homisinteg}, we deduce that $\exp(-2k')$ is integrable. Hence we can suitably normalize and get a probability measure whose density is proportional to $\exp(-2k'(x))$. Now, we can apply Proposition \ref{invariantdist} to claim that the unique reversible invariant distribution of the process $X'$ exists and is given by the unnormalized density $\exp(-2k'(x))$.

Now, consider a projection matrix $A$ with range space $H$. Thus, $A^{\perp}=I-A$ is a projection onto $H^{\perp}$. Now, by assumption, the function $k$ splits additively as
\[
\begin{split}
k(x) &= k_1(Ax) + k_2(A^{\perp}x)
\end{split}
\]
Taking gradient on both sides, one obtains
\[
\begin{split}
\Rightarrow \; \nabla k(x) &= A'\nabla k_1(Ax) + \left(A^{\perp}\right)' \nabla k_2(A^{\perp}x).
\end{split}
\]
Here $A'$ and $(A^{\perp})'$ refers to taking adjoints. But $A$ being a projection is self-adjoint and satisfies $A^2=A$. Thus, it follows that
\begin{equation}\label{newgrad}
A\nabla k(x) = A^2\nabla k_1(Ax) + A(I-A)\nabla k_2(A^{\perp}x)= A\nabla k_1(Ax).
\end{equation}
Exactly in the same way we also have
\begin{equation}\label{newgrad2}
A\nabla k'(x) = AA'\nabla k_1(Ax) + A(A^{\perp})'\nabla k'_2(A^{\perp}x)= A k_1(Ax).
\end{equation}

Now, consider the processes $Y=AX$ and $Y'=AX'$ which clearly satisfy the following differential equations
\begin{equation}
\begin{split}
\dx Y_t &= -A\nabla k(X_t)\dx t + A\dx W_t,\quad \text{and} \\
\dx Y'_t &= -A\nabla k'(X'_t)\dx t + A\dx W_t.
\end{split}
\end{equation}
Using equations \eqref{newgrad}, and \eqref{newgrad2} we can rewrite the above equations as
\[
\begin{split}
\dx Y_t &= -A\nabla k_1(Y) + A\dx W_t,\quad \text{and}\\
\dx Y'_t &= -A\nabla k_1(Y') + A\dx W_t.
\end{split}
\]
Clearly the laws of the process $Y$ and $Y'$ are identical due to the uniqueness in law of the weak solutions of their stochastic differential equation. To get past the arbitrariness of the hyperplane $H$, one can simply observe that if dimension$(H)=d \le n$, there exists a $(d\times n)$ matrix $D$ which is a bijection between $H$ and $\rr^d$. The laws of $DY$ and $DY'$ are identical by standard theory of SDE. Now one simply inverts $D$ onto $H$ to obtain our conclusion.

Now, since $A$ is a linear map, the process $Y'$ has a unique invariant distribution $\nu$ induced by the invariant distribution of $X'$, and given by 
\begin{equation}\label{whatisnu}
\nu(B) = \mu'(Ax\in B),\qquad \forall \; B \in \mcal{B}(H),
\end{equation}
where $\mu'$ is the unique invariant distribution for the process $X'$. Thus, by the equality in law of the processes $Y$ and $Y'$ it follows that the process $Y_t=AX_t$ has a unique invariant distribution given by $\nu$ according to the recipe above.

It is easy to give an intrinsic description of $\nu$ from \eqref{whatisnu}. If $X'$ is distributed according to $\mu'$, then for any $B \in \mcal{B}(H)$, we have 
\[
\begin{split}
\nu(B)=\mu'\left(AX'\in B\right)&= \int_{Ax\in B}\exp\left(-2k_1(Ax) - 2k_2'(x-Ax)\right)\dx x\\
& =\int_{y\in B}\exp\left(-2k_1(y)\right)\dx y \int_{H^{\perp}}\exp\left(-2k_2'(z)\right)\dx z.
\end{split}
\]
Clearly, $\nu$ has a density proportional to $\exp(-2k_1(y))$ with respect to the Lebesgue measure restricted to the hyperplane $H$.
\end{proof}

\section{Simplicial cones and Exponential distributions}\label{cones}

We have seen in the last section that the invariant probability distributions for the SDEs described in \eqref{sde1} have unnormalized densities with respect to the Lebesgue measure given by $\exp(-2k(x))$, where $k$ is a nonnegative, irreducible, positively homogeneous function. Even with such an explicit description of the density function it can be very hard to compute any means, variances, or one-dimensional marginal distributions. Our objective in this section is to link them to the combinatorial structures of the unit ball generated by the positively homogeneous function $2k$, particularly when these are star-shaped or convex polytopes. In the special case of simplicial polytopes this allows us to furnish a complete description of the invariant measure in terms of independent Exponential random variables.

The main geometric idea is the following. Consider, as before, a drift function $b$ which is constant over cones $C_1, C_2, \ldots, C_r$. We assume that $b$ is the negative gradient in the sense of distributions of the function $k(x)=-\iprod{x,b(x)}$. Thus, $2k(x)$ is a linear function inside each cone $C_i$.

Suppose now that $C_i$ is simplicial. Simplicial cones are those that can be transformed to the positive quadrant by applying a non-singular linear transformation. That is to say, for each $C_i$ there exists $n$ linearly independent vectors in $\rr^n$, denoted by $\beta_1, \beta_2, \ldots, \beta_n$, such that $C_i=\left\{ x\in \rr^n: \; x=\sum_i a_i \beta_i, \; a_i \ge 0\; \forall\; i \right\}$. It is now not hard to see that in that case the probability measure given by $\exp(-2k(x))$ restricted to the cone $C_i$ must be a linear transformation of independent Exponential distributions. The difficulty in the execution of the previous argument is to identify from the function $b$ that a simplicial polytope is lurking behind the scenes and to compute the necessary linear transformations. In the following two subsections we consider two general classes of examples where the argument can be fully carried out. One, the regular case, is where each $C_i$ can be mapped to any other by a group of orthogonal transformations. This leads us to a connection with finite irreducible Coxeter groups. The other, which is not regular, deals with graphs and the combinatorics hidden in their structure.

\subsection{The regular case: groups of orthogonal transformations}\label{groupcase}

Consider a finite dimensional Euclidean vector space $V$. A linear transformation from $V$ to itself is called orthogonal if the corresponding matrix is orthogonal. In that case the determinant of the transformation is $\pm 1$. An important class of orthogonal matrices is given by reflections. A reflection along a unit vector $u\in V$ corresponds to the matrix $I - 2uu'$ (all vectors are columns and $u'$ denotes the transpose of $u$). Geometrically it produces the mirror image of any vector with respect to the hyperplane orthogonal to $u$. A reflection group is a group of matrices, each element of which is a reflection matrix.

Let $G$ be a group of orthogonal matrices. $G$ is called irreducible if there is no non-trivial subspace $W$ of $V$ which is stable under the action of $G$, i.e., $\rho(s)W \subseteq W$, for all $s\in G$. 
As a recurring example, considering the symmetric group of permutations on $n$ elements. It has a natural representation as permutation matrices on $\rr^n$. This is not irreducible since the one-dimensional subspace $W_1$ spanned by the vector of all ones remains invariant under the action of the group. However, the action restricted to $W_1$ and $W_2=W_1^\perp$ is irreducible.

For a finite group of orthogonal linear transformations $G$ we now define what is known as a \textit{fundamental region}. Please see Chapter 3 in \cite{reflectiongroups} for more details.

\begin{fundamental} A subset $F\subseteq V$ is known as a fundamental region for a group of orthogonal transformations $G$ if 
\begin{enumerate}
\item $F$ is open,
\item $F\cap AF=\emptyset$ if $A \neq I$, where $I$ is the identity matrix and $A \in G$,
\item $V= \cup_{A\in G} \overline{AF}$, where $\overline{B}$ denotes the topological closure of a subset $B$. 
\end{enumerate}
\end{fundamental}

For the representation of the symmetric group as permutation matrices such a region is provided by the cone $\left\{ x\in \rr^n: \; x_1 < x_2 < \ldots < x_n \right\}$. In general, fundamental regions are not unique. 
 
Henceforth we will work with $V = \rr^n$ even though proper subspaces of $\rr^n$ provide another rich class of examples.   
  
The groups we will be interested in, which includes the permutations as a special case, are generated purely by reflections. When irreducible, these groups are known as Coxeter groups and we review their basic structure below.

A reflection along a vector $r$ is uniquely characterized by the fact that it keeps every vector orthogonal to $r$ unchanged and flips the sign of every multiple of $r$. A Coxeter group is a finite irreducible group of orthogonal transformations generated by finitely many reflections. A frequent class of examples are the Dihedral groups which are the symmetry groups of regular polygons. Although, these groups contain both rotations (think of a unit square being rotated by $\pi/2$) and reflections (the square getting reflected along the mid-axis), one can show that these groups can be generated purely by the reflection elements (reflecting the square twice along different axes amounts to a rotation).

Closely associated with the Coxeter groups is the idea of \textit{root systems}. If $G$ contains a reflection along $r$, then both $r$ and $-r$ are known as \textit{roots} of $G$. Let $\Delta$ be the set of all the roots of $G$, usually referred to as the \textit{root system} of the Coxeter group. Recall the definition of a fundamental region. We are going to construct a fundamental region for $G$ which is a simplicial cone.  

Let $G$ be a Coxeter group acting on $\rr^n$. Thus, in particular, it is irreducible. 

Choose any vector $u\in \rr^n$ such that $\iprod{u,r}\neq 0$ for any root $r$ of $G$. Then the root system is partitioned into two subsets
\[
\Delta^+_u =\{r\in \Delta:\; \iprod{r,u} > 0 \}, \qquad \Delta^-_u= \{ r\in \Delta:\; \iprod{r,u} < 0\}.
\] 
Theorem 4.1.7 in \cite{reflectiongroups} (and several lemmas preceding it) proves the following result. 

\begin{thm}
There is a unique collection of $n$ many vectors $\Pi_u$ in $\Delta_u^+$ such that every vector in $\Delta_u^+$ can be written as a linear combination of elements in $\Pi_u$ with non-negative coefficients. This collection, known as the $u$-base or fundamental roots, and denoted by $\Pi_u$, is linearly independent and forms a basis of $\rr^n$. 
\end{thm}

A $u$-base provides a fundamental region for a Coxeter group $G$, by defining
\eq\label{fundacox}
F_u := \left\{ x\in \rr^n:\; \iprod{x,r} >0 \quad \forall\; r\in \Pi_u \right\}.
\en
That this is indeed a fundamental region has been proved in Theorem 4.2.4 in \cite{reflectiongroups}. That this is a simplicial cone follows since the vectors in $\Pi_u$ are linearly independent.  
\medskip

Let us now return to the framework in Proposition \ref{invariantdist}. We start with a drift function $b:\rr^n \rightarrow \rr$ that is scale invariant, i.e., $b(\alpha x)= b(x)$ for all $\alpha > 0$. We would like to analyze the probability measure given by normalizing $\exp(-2k(x))$ where $k(x)=-\iprod{x,b(x)}$. Suppose the drift function $b$ takes finitely many values on $\rr^n$ and satisfies the property that $b(Ax)=Ab(x)$ for all $A\in G$ for some Coxeter group $G$. Can we claim that there is a finite sequence of non-overlapping simplicial cones whose union is the whole space and such that $b$ takes a constant value inside each cone ?

The answer is no in general. However, there is a simple sufficient condition which indeed guarantees an affirmative answer to the question. Let, as before, $k(x)=-\iprod{x,b(x)}$. If $k$ is non-negative, then clearly, by positive homogeneity, $k$ is the Minkowski functional of a star-shaped body containing the origin. The question in the previous paragraph is equivalent to asking whether this star-shaped body can be triangulated as a disjoint union of simplices each of which contains the origin as a extreme point. We are going to show below that the answer to the question is yes, if $k(x)=\max_{A\in G}\iprod{\lambda,Ax}$ for some vector $\lambda$. Seen through a geometric lens, this is equivalent to the statement that the star-shaped body generated by $k$ is a convex polytope which then turns out to be simplicial.

\begin{lemma}\label{constructk}
Let $G$ be a finite irreducible group of orthogonal matrices on $\mathbb{R}^n$. For any $\lambda \in \rr^n$, $\lambda \neq 0$, the function
\begin{equation}\label{whatisk}
k(x)= \max_{A\in G}\iprod{\lambda,Ax}, \quad x\in \rr^n,
\end{equation}
is a nonnegative positively homogeneous function on $\rr^n$, which is irreducible. Moreover, $k$ is invariant under the action of the group. That is $k(x) = k\left(Ax\right)$, for all $A\in G$, and for all $x\in \rr^n$.
\end{lemma}

\begin{proof}
It is trivial to see that $k$ is positively homogeneous. To show that it is nonnegative we use the fact (see \cite{diaconisbook}) that for any non-trivial irreducible group $G$ of orthogonal matrices, the sum $\sum_{A \in G} A $ is the zero matrix. Since $G$ is irreducible, it follows that
\[
\sum_{A\in G} \iprod{\lambda,Ax} = \iprod{\lambda,\sum_A Ax} = 0.
\]  
But this implies that the maximum must be non-negative. Thus $k$ is nonnegative. 

To prove that $k$ must be strictly positive for all nonzero vectors, note that for the previous argument, $k(x)=0$ for some $x$ would imply that
\[
\iprod{\lambda, Ax}=0, \quad \forall\; A\in G.
\] 
Define the subspace
\begin{equation}\label{whatisvl}
V_{\lambda}= \left\{ x\in \rr^n:\; \iprod{\lambda, Ax}=0 \quad \forall\; A\in G   \right\}.
\end{equation}
Next, note that $V_{\lambda}$ is stable under the action of $G$. To see this take any $B\in G$ and any $x\in V_{\lambda}$, then clearly $Bx \in V_{\lambda}$. Thus $AV_{\lambda} \subseteq V_{\lambda}$ for all $A\in G$. But, since $G$ is irreducible, $V_{\lambda}$ must be either zero or the entire space. If $V_{\lambda}$ is the entire subspace, then by putting $A$ to be the identity in the definition \eqref{whatisvl}, we get $\iprod{\lambda,x}=0 \;\forall\; x\in V$. This shows that $\lambda$ must be zero, which we have ruled out in our assumption. 
The invariance of $k$ under the action of $G$ is clear by the homomorphism property.\end{proof}

The last lemma proves that $k$ is the gauge function of a convex polytope containing the origin. The following lemma shows that the number of extremal faces of the polytope is given by the size of the orbit of $\lambda$.

\begin{lemma}\label{maximumwhere}
Consider $k$ as in Lemma \ref{constructk}. Given any $\lambda$, there exists $x\neq 0$ such that $k(x)=\iprod{\lambda,x}$. Moreover, $k\left( Ax \right) = \iprod{A\lambda, Ax}$ for all $A\in G$.

\end{lemma}

\begin{proof}
To prove the first assertion, suppose that for all $x$, we have
\[
\iprod{\lambda,x} < k(x)= \max_{A\in G} \iprod{A\lambda,x}.
\] 
Then, for any such $x$, for any $B\in G$, we also have
\[
\begin{split}
\iprod{B\lambda, x} &= \iprod{\lambda,B' x} < k\left( B'x\right)=\max_{A\in G}\iprod{A\lambda, B'x}=\max_{A\in G}\iprod{BA\lambda,x} = k(x).
\end{split}
\]
But that would imply $\max_{A\in G}\iprod{A\lambda, x} < k(x)$ which is clearly a contradiction.

For the second assertion, consider $x,\lambda$ such that $k(x)=\iprod{x,\lambda}$. Now
\[
k\left( Bx \right) = \max_{A\in G}\iprod{A\lambda,Bx}= \max_{A\in G}\iprod{B'A\lambda,x}=\max_{A\in G}\iprod{(B^{-1}A)\lambda,x}.
\]
The right hand side is maximized when $B^{-1}A=I$ which proves the lemma.
\end{proof}

We now show that for Coxeter groups that the polytope generated by $k$ is simplicial. That is to say, all its extremal facets are simplices.

\begin{lemma}\label{maxsimp}
Consider any $n$-dimensional irreducible group of orthogonal matrices $G$. Let $\lambda \in \rr^n$ be such that $A\lambda \neq \lambda$ for all $\lambda \neq e$. In other words, $\lambda$ has no non-trivial stabilizer. 

Consider the region $\mcal{C}= \left\{ x\in \rr^n:\; k(x)=\iprod{\lambda,x}  \right\}$. Then the interior of $\mcal{C}$, given by 
$\mcal{C}_0 = \left\{ \iprod{\lambda,x} > \iprod{A\lambda,x}, \; \forall \; A \in G \right\}$, provides a fundamental region for the group. Additionally, if $G$ is a Coxeter group, $\mcal{C}$ is an $n$-dimensional closed simplicial cone. 
\end{lemma}

\begin{proof}
Note that $\mcal{C}$ is the region $\left\{x\in \rr^n:\; \iprod{\lambda, x} \ge \iprod{A\lambda, x}\; \forall A\in G  \right\}$. We first show that $\mcal{C}$ is a $n$-dimensional convex cone. Label the non-identity elements of the group $G$ by $A_1,A_2,\ldots,A_N$ where $N+1=\lvert G\rvert$. 

Consider the matrix $Q$ defined by
\eq\label{livectors}
Q = \left[\begin{array}{c}
\lambda - A_1\lambda\\
\lambda - A_2\lambda\\
\vdots \\
\lambda - A_N\lambda\\
\end{array}\right],
\en
where all vectors are row vectors. This matrix $Q$ when applied to vectors of $\mcal{C}$ produces nonnegative entries. The dimension of $Q$ is $N \times n$. We first show that the rank of $Q$ is $n$. Note that, trivially the rank cannot be more than $n$. We show that the dimension of the kernel is zero which proves that the rank must be exactly $n$. 

Let $\mcal{K}$ denote the kernel, $\left\{ x:\; Qx=0 \right\}$. Then we claim that $\mcal{K}$ is invariant under the action of the group. This is because, $x\in \mcal{K}$ iff $\iprod{\lambda,x}=\iprod{A\lambda,x}$ for all $A\in G$. But, for any $B,A\in G$, we also have
\[
\begin{split}
\iprod{A\lambda, Bx}&= \iprod{B^{-1}A\lambda, x}=\iprod{\lambda,x},\quad \text{since}\; x\in \mcal{K},\\
&=\iprod{B^{-1}\lambda, x}= \iprod{B'\lambda,x}= \iprod{\lambda, Bx}.
\end{split}
\]
Thus $Bx\in \mcal{K}$. This proves that $\mcal{K}$ stable under the action of the group. But since the representation is irreducible, this implies that $\mcal{K}$ is either zero or the full space. But, it is easy to see that if $\mcal{K}$ is the full space, then $\lambda$ must be zero. This proves that $\mcal{C}$ is $n$-dimensional. That it is a convex cone is obvious.

Since the dimension of $\mcal{C}$ is $n$ and the stabilizer of $\lambda$ is the identity, the interior of the cone is given by $\mcal{C}_0$. We now show that $\mcal{C}_0$ is a fundamental region for $G$ by verifying the definition. $\mcal{C}_0$ is open by definition. For any $A\in G $, $A\neq I$, note that $A\mcal{C}_0$ is the following set
\[
\begin{split}
\left\{Ax,\; x\in \mcal{C}_0\right\}&= \left\{y:\;\iprod{\lambda, A^{-1}y} > \iprod{B\lambda, A^{-1}y},\; \forall B\in G \right\}\\
&= \left\{y: \iprod{A\lambda, y} > \iprod{B\lambda,y},\; \forall B \in G  \right\}.
\end{split}
\]
Thus $x\in \mcal{C}_0 \cap A\mcal{C}_0$ implies $\iprod{\lambda,x} > \iprod{A\lambda, x} > \iprod{\lambda, x}$ which is impossible. Thus the intersection must be empty. It is also trivial to see that $\cup_{A\in G}A \mcal{C}= \rr^n$. This shows that $\mcal{C}_0$ is a fundamental region.

For Coxeter groups we now show that $\mcal{C}_0$ is the same region as $F_{\lambda}$ defined in \eqref{fundacox}. Notice first that if $A$ is a reflection along a vector $r$ for some $A\in G$, then
\eq\label{bothfunda}
\iprod{\lambda - A\lambda , x}= 2\frac{\iprod{r,\lambda}\iprod{r,x}}{\norm{r}^2}.
\en
Now, suppose $x\in \mcal{C}_0$. Then $\iprod{\lambda - A\lambda, x} > 0$ for all non-identity $A\in G$, in particular, for all $A$ which corresponds to reflections along the roots. Thus, for any root $r\in \Delta_{\lambda}^+$, from the above equality we get that $\iprod{r,x} > 0$. From the definition of $F_{\lambda}$, it is now obvious that $x\in F_{\lambda}$. Thus we have shown that $\mcal{C}_0 \subseteq F_{\lambda}$.

For the reverse equality, note that if $\mcal{C}_0$ is a proper subset of $F_{\lambda}$, then for every $A\in G$, the set $A\mcal{C}_0$ is a proper subset of $AF_{\lambda}$. But, each $AF_{\lambda}$ is disjoint and the union of the closures of $A\mcal{C}_0$ is the entire $\rr^n$. This is clearly impossible. Thus, we have shown that equality holds among the two fundamental regions $\mcal{C}_0$ and $F_{\lambda}$. Since $F_{\lambda}$ is a simplicial cone, so is $\mcal{C}_0$. Thus $\mcal{C}$ is a closed simplicial cone. 
\end{proof}

The connection between simplicial cones and Exponential distributions is made precise in the next lemma. 

\begin{lemma}\label{simpexp}
Consider a sequence of simplicial cones $C_1, C_2, \ldots, C_r$ which are open, disjoint, and the closure of their union is the whole space.
Let $k$ be a nonnegative, irreducible, positively homogeneous function such that $k$ is linear on each $C_i$. That is, $k(x)= \iprod{\lambda_i,x}$, {for all} $x\in C_i$, for some sequence of vectors $\lambda_1,\lambda_2,\ldots,\lambda_r$ which may not be all distinct.

Let $X$ be a random variable whose density with respect to the Lebesgue measure on $\rr^n$ is proportional to $e^{-2k(x)}$. Let $B_i$, $i=1,2,\ldots,r$, be any set of invertible matrices such that for each $i$, the matrix $B_i$ maps the cone $C_i$ onto the $n$-dimensional quadrant. Also, let $\alpha_1(i),\ldots,\alpha_n(i)$ be the coefficients in the unique representation $\lambda_i'= \alpha'(i) B_i$. Then, the random vector $Y=(Y_1,Y_2,\ldots,Y_n)$, where
\[
Y_j:= \alpha_j(i)\iprod{(B_i)_{j*}, X}, \quad \text{if}\; X\in C_i,
\]
is a vector of iid Exponential$(2)$ random variables. Here $(B_i)_{j*}$ denotes the $j$th row of the matrix $B_i$.
\end{lemma}

\begin{proof}
Since the cones $C_i$'s are simplicial, the existence of the matrices $B_i$'s follows from the definition. Moreover, it follows from the definition of $\alpha(i)$ that if $x\in C_i$ and $z=B_ix$, then $k(x)=\iprod{\lambda_i,x}=\sum_j \alpha_j(i) z_j$ irrespective of $i$. Let $B_i^*$ be the matrix given by
\[
\left(B_i^*\right)_{j*}= \alpha_j(i)(B_i)_{j*}, \quad j=1,2,\ldots,d.
\]
Since the transformations are piecewise linear, it follows that 
\[
Y:= \sum_j B^*_j X 1\left( X\in C_j \right)
\] 
has a density proportional to $\exp\{ -2\sum_j y_j\}$ over the quadrant $\{y:\; y_1\ge 0, y_2\ge 0, \ldots, y_d\ge 0 \}$. This immediately identifies itself as the joint density of iid exponentials with rate two. This proves part (2). 
\end{proof}

\begin{prop}\label{mainthmstatic}
Let $G$ denote a Coxeter group acting on $\rr^n$. For $\lambda\in \rr^n$, $\lambda\neq 0$, let $k(x)=\max_{A\in G}\iprod{A\lambda, x}$. 
Let $\nu$ be the probability measure with unnormalized density $\exp\{-2k(x)\}$ on $\rr^n$. Then the following statements hold true.

\begin{enumerate}
\item  When the stabilizer of $\lambda$ is trivial, the conic hull of the finite set $\{\lambda - A\lambda,\; A\in G \}$ contains $n$ linearly independent generating vectors $\{\eta_1,\eta_2,\ldots,\eta_n\}$. That is, every other vector in the set can be expressed as a linear combination of the generators with nonnegative coefficients.  

\item Let $X$ denote a random variable with distribution $\nu$. Also let $\alpha_i$ denote the unique positive  coefficient of $\eta_i$ in the expansion 
\begin{equation*}\label{whatisalpha}
\lambda= \sum_{i=1}^n \alpha_i \eta_i.
\end{equation*}
For $i=1,2,\ldots,n$, define the change of variable
\begin{equation*}\label{howtochange}
Y_i= \alpha_i\iprod{A\eta_i,X}, \quad \text{when}\quad k(X)=\iprod{A\lambda, X}.
\end{equation*}
Then, the vector $Y=(Y_1,Y_2,\ldots,Y_n)$ are iid Exponential$(2)$ random variables.
\end{enumerate} 
\end{prop}

\begin{proof}
To prove the first assertion, we use Lemma \ref{maxsimp}. Assume first that the stabilizer of $\lambda$ is trivial. Then, Lemma \ref{maxsimp} tells us that the cone 
\eq\label{pfcone}
\mcal{C}=\{x:\; \iprod{\lambda-A\lambda,x} \ge 0  \}
\en
is an $n$-dimensional simplicial cone. Hence there exists exactly $n$ many linearly independent generators among the set $\{\lambda - A\lambda, \; A\in G \}$ such that every other vector is a linear combination with nonnegative coefficients. The rest of the result follows directly from Lemma \ref{simpexp}. Notice that the coefficients of $\lambda$ in the expansion $\lambda= \sum_{i=1}^n \alpha_i \eta_i$ are positive by Farkas lemma. This is because, for any non-zero $y\in \mcal{C}$, the inner product $\iprod{\lambda,y}= k(y) > 0$ by irreducibility and nonnegativity of $k$. 
\medskip

When $\lambda$ has a nontrivial stabilizer, the cone in \eqref{pfcone} is a union of several simplicial cones. The simplest way to see this is to take a sequence $\lambda_l$ which have no nontrivial stabilizers and which converges to $\lambda$. The component cones are then given by the limits of the sequence of simplicial cones generated by them. In any case Lemma \ref{simpexp} still holds, however, the vectors $\{\eta_1, \ldots, \eta_n\}$ have to be determined by the limiting procedure. 
\end{proof}

The proof of Proposition \ref{gpdiff1} in the Introduction now follows easily.

\subsection{Examples}\label{examples}

Let us consider some examples of consequences of Proposition \ref{gpdiff1}. 
\medskip

\noindent\textbf{Example 1: Rank based interactions.} Brownian motions with rank based interactions have been considered in equation \eqref{ranksde}. Clearly the drift function $b$ is constant over finitely many cones determined by the permutation generated by the ordered coordinates. Let $x[1] \le x[2] \le \ldots x[n]$ denote the coordinates of an $n$ dimensional vector arranged in increasing order. It is easy to see that
\[
k(x)= \sum_{i=1}^n \delta_i x{[i]},
\] 
is a positively homogeneous function which is not irreducible since it takes a constant value over the linear span of the vector of all ones. However, if we let $H$ be the subspace orthogonal to $1$, then $k$ splits additively as
\[
k(x) = -\sum_{i=1}^n \delta_i \left( x{[i]} - \bar{x}\right) + n\bar{\delta}\bar{x}.
\]
Let $k_1:H\rightarrow \rr$ denote the restriction of $k$ to $H$, then it is clear that $k_1(x)=k_1(A_{\sigma}x)$ for any permutation matrix $A_{\sigma}$. Now, the group of permutation matrices acting on $H$ is well known to be irreducible and generated by reflections along $e_{i+1}-e_i$ for $i=1,2,\ldots,n-1$. This is just a restatement of the fact that every permutation can be written as a product of transposes. Thus, it is a Coxeter group, often denoted by $\mcal{A}_{n-1}$. Thus the conclusions of Proposition \ref{constructk} applies and $k_1$ is irreducible if
\eq\label{whenkrank}
k_1(x)= \max_{A_{\sigma}}\iprod{-A_{\sigma}\delta,x}.
\en
This condition is equivalent to the condition $\delta_1 \ge \delta_2 \ge \ldots \ge \delta_n$.

Hence, from Proposition \ref{additivesplit} it also follows that the projection $P_HX_t$ of the diffusion on to $H$ has an invariant distribution whose density with respect to Lebesgue measure on $H$ is proportional to $\exp(-2k_1(x))$. 

We now apply Proposition \ref{mainthmstatic}. One can see that the set of vectors $-\delta + A_{\sigma}\delta$, as $\sigma$ ranges over permutations, contains positive multiples of vectors $e_{i+1} - e_i$, since they correspond to the transposition of $i$ and $i+1$. These $n-1$ linearly independent vectors are the conic extremes of the set. Thus, by Proposition \ref{mainthmstatic}, the spacings $X[i+1]- X[i]$ are independent Exponential random variables under the invariant distribution. The correct rates can be easily verified.

In this example it is easy to see the shortfall of the sufficient condition \eqref{whenkrank}. The drift function is constant over a fundamental region $F=\{x\in H:\; x_1 < x_2 < \ldots < x_n \}$ which is clearly simplicial. Now $k(x)$ is irreducible if and only if the unit ball generated by $k$ is compact. By symmetry, we can restrict our attention to $F$. Since $F$ is simplicial we can apply a suitable linear transformation to map it to the positive quadrant. Thus, it can be easily verified that if $\bar{\delta}$ denoted the average of the coordinates of $\delta$, the intersection of the unit ball with $F$ is compact if and only~if
\[
\alpha_k:=\sum_{i=1}^k\left( \delta_i - \bar{\delta} \right) >0, \quad \forall\; i=1,2,\ldots, n-1.
\]
This is precisely the condition derived by Pal \& Pitman in \cite{palpitman} using the theory of reflected Brownian motions and is weaker than the sufficient condition that the coordinates of $\delta$ decreases with increasing values of the coordinates of $x$.   
\medskip

\noindent\textbf{Example 2: Sign-rank based interactions.} 

An example of interactions similar to rank-based can be generated by allowing both the rank and signs coordinates to determine the drift. As before, we start with the $n$-dimensional SDE:
\eq\label{gensgnrank}
d X_t = b(X_t) dt + d W_t, 
\en
where $W_t$ is an $n$-dimensional Brownian motion. Suppose that the drift function takes finitely many values, is scale invariant, and $b(Ax)=Ab(x)$ whenever $A$ is either a permutation matrix or a diagonal matrix with each diagonal entry being plus or minus one. Thus, not only that the values of the drift get permuted whenever the coordinates get permuted, but also the sign of the drift changes with the sign of the corresponding coordinate. 

The group generated by the collection of permutation matrices and the diagonal matrices of sign flips is a Coxeter group denoted by $\mcal{B}_n$. Please see pages 66--71 of \cite{reflectiongroups} for more details. We can safely apply Proposition \ref{gpdiff1}. Thus, the $n$ dimensional process under such a sign-rank based interaction is recurrent if there is a vector $\lambda \in \rr^n$, $\lambda \neq 0$, such that
\[
k(x)=-\iprod{x, b(x)} = \max_{A\in \mcal{B}_n}\iprod{A\lambda, x}.
\]
If we restrict the above condition to the cone $\{ x:  0 < x_1  < x_{2}  < \ldots < x_n \}$ we see that the vector of drifts $b(x)=-\lambda$ where $\lambda$ satisfis that $0 \le \lambda_1 \le \lambda_2 \le \ldots \le \lambda_n $. 

When this does hold true, $X$ has an unique long term stationary distribution. To find the decomposition of this probability distribution in terms of  independent Exponentials we consider a $\delta$ all of whose coordinates are non-zero and distinct. That is, it has a trivial stabilizer subgroup in $\mcal{B}_n$. Consider the conic hull generated by the set of vectors $\{ \lambda - A\lambda, \; A\in \mcal{B}_n \}$. As in the case of rank-based interactions one can see that the generators of the conic hull are positive multiples of the vectors $e_1$ and $\{e_{i+1} - e_i , \; i=1,2,\ldots n-1\}$. 

Now we apply the final conclusion of Prop \ref{gpdiff1}. To get the vector of Exponentials under the invariant distribution, note that
\[
k(x) = \sum_{i=1}^n \lambda_i \abs{x}[i],
\]
where $\abs{x}[1] \le \abs{x}[2] \le \ldots \abs{x}[n]$ are the ordered values of the absolute values of the coordinates $(\abs{x_1}, \abs{x_2}, \ldots, \abs{x_n})$.

Thus, from Proposition \ref{gpdiff1} it follows that the random vector $(\abs{X}[1], \abs{X}[i+1] - \abs{X}[i], \; i=1,2,\ldots, n-1)$ are distributed as independent Exponentials. 

To compute the rates of these Exponentially distributed random variables, one needs to compute the coefficient of $\lambda$ with respect to the basis vector $e_1$ and $\{e_{i+1} - e_i , \; i=1,2,\ldots n-1\}$. A simple computation leads us to the conclusion that the corresponding vector of rates of these Exponentials are given by 
\[
\left(2\sum_{s=j}^n \lambda_s, \quad j=1,2,\ldots, n\right).
\]
\medskip

\noindent{\textbf{Example 3: Constrained sign-rank based interactions.}}  
  
This is a interesting class of constrained sign-rank based interactions where not all sign changes of coordinates affect the drift vector. Consider again the stochastic differential equation \eqref{gensgnrank}. Suppose that the drift function $b$ takes finitely many values, is scale invariant, and $b(Ax)=Ab(x)$ for all permutation matrices $A$ and all diagonal matrices with each diagonal entry being positive or negative one \textit{with the constraint} that there are even number of negative ones. Thus, the sign of the drift vector changes when either the ordering of coordinates change or when pairs of coordinates have flipped their signs. 

The groups generated by the permutation matrices and the diagonal matrices with even number of sign flips is clearly a subgroup of the $\mcal{B}_n$. They form, in fact, a family of Coxeter groups, usually denoted by $\mcal{D}_n$ where $n$ denotes the dimension of the underlying space. We again refer the reader to pages 66--71 of \cite{reflectiongroups} for more details. 

We apply Proposition \ref{gpdiff1}. Thus, the $n$ dimensional process under such a constrained sign-rank based interaction is recurrent if there is a vector $\lambda \in \rr^n$, $\lambda \neq 0$, such that
\[
k(x)=-\iprod{x, b(x)} = \max_{A\in \mcal{D}_n}\iprod{A\lambda, x}.
\]

The above condition is more difficult to analyze than the previous examples. One can show using known results about the fundamental root systems of $\mcal{D}_n$ (page 71 in \cite{reflectiongroups}) that the drift is determined by the fact that over the cone 
\eq\label{concons}
 \left\{ x:  0 < x_1 + x_2, \quad  x_1  < x_{2}  < \ldots < x_n \right\}, 
\en
the drift is a constant $b(x)=-\lambda$ where $\lambda$ satisfis that $0 \le \lambda_1 + \lambda_2$ and $ \lambda_1 \le \lambda_2 \le \ldots \le \lambda_n $. This cone is actually a fundamental region for the group $\mcal{D}_n$. Thus, the drift vector is now determined over entire $\rr^n$ by the property $b(Ax)=Ab(x)$ for all $A\in \mcal{D}_n$.

Under this condition the process has a long term stationary distribution. To find what functions turn out to be independent Exponentials, we need to understand, for a given $x$, what unique $A\in \mcal{D}_n$ will produce $k(x)=\iprod{\lambda, Ax}$. Clearly, this will happen if $A$ is chosen such that $Ax$ belongs to the cone \eqref{concons}. There are two cases to consider. One, when the number of coordinates of $x$ that are negative is even. In this case, one simply flips the signs of these coordinates, and then ranks the absolute values to get a vector in \eqref{concons}. Both these actions are permissible since they correspond to multiplication by matrices in $\mcal{D}_n$. The other case is when $x$ has odd number of negative coordinates. First, one has to flip the sign of all the negative coordinates except the least negative one and then rank all the coordinates. In this ordering, either the absolute value of the second least negative coordinate is less than the least positive coordinate in which case, the resulting vector is in \eqref{concons}. Or, it is not, in which case we need to compare the least negative coordinate with the least positive coordinate. Their sum is either positive or negative, and we make the appropriate (zero or two) sign flips to get the right transformation. Let $H(x)$ be the resulting vector produced by the above procedure.

Under this stationary distribution, the vector of random vector $H(X_1, X_2, \ldots, X_n)$ is distributed as $n$ independent Exponentials. Furthermore, as in the case of sign-rank interactions, one can work out the linear algebra to compute that the corresponding vector of rates of these Exponentials are given by 
\[
\left(\sum_{s=1}^n \lambda_s,\quad -\lambda_1+ \sum_{s=2}^n \lambda_s,  \quad \quad 2\sum_{s=j}^n \lambda_s, \quad j=3,4,\ldots, n\right).
\]

\subsection{An example of irregular interaction}\label{graphcase} 
    
In this subsection we consider an example of interacting Brownian motions with drift functions that are still piecewise constant on cones, but are not consistent with any group action.    

Consider a graph $\G$ on $n$ vertices where the vetices are labeled by $\{1,2,\ldots,n\}$. The edge between $i$ and $j$ have an associated edge weight $\beta_{ij}$, which is zero if there is no edge between the two vertices. 

Consider the SDE on $\rr^n$ given by
\eq\label{sdegraph}
dX_t(i) = \sum_{j=1}^n\beta_{ij}\text{sign}\left( X_t(j) - X_t(i) \right)dt + dW_t(i),
\en  
where, as before, $W=(W(1), W(2), \ldots, W(n))$ is an $n$-dimensional Brownian motion. 
  
When all the edge weights are nonnegative, the model can be described by saying that the Brownian motions, which are indexed by the vertices of the graphs, get attracted towards one another. The constants $\beta_{ij}$ measure the strength of their attraction.    
  
The appropriately defined drift function $b(x)$ is piecewise constant on the family of cones $C_{\pi}:=\{ x: x_{\pi(1)}\le x_{\pi(2)}\le \ldots \le x_{\pi(n)} \}$, where $\pi$ ranges over all permutations of $n$ labels. However, it might not satisfy the condition that $b(A_{\pi}x)=A_{\pi}b(x)$ where $A_{\pi}$ is the permutation matrix corresponding to $\pi$.

Note that the drift function $b$ is the negative of the gradient of the positively homogeneous function
\[
k(x)=\sum_{i<j}\beta_{ij}\abs{x_i-x_j}, \qquad \forall \; x\in \rr^n.
\]
It can be easily verified (and intuitive) that if $\beta_{ij}$'s are nonnegative and $\G$ is connected, the function $k$ is irreducible when restricted to the subspace $H$ orthogonal to the vector of all ones (which we denote by $1$). Then the conclusions of Proposition~\ref{additivesplit} applies. In particular, if we define
\[
k'(x) = k(x) + \abs{\iprod{x,1}}
\] 
then $k'$ is integrable and both the probability measures induced by functions $\exp(-2k)$ and $\exp(-2k')$ on $H$ must be the same. 

Assume $\beta_{ij}\ge 0$ and $\G$ is connected. For convenience absorb the factor of two in $\exp(-2k)$ in the definition of $\beta$. Let $\mu_n$ be the probability measure whose unnormalized density is given by  
\eq\label{invgraph}
\exp\left(-\sum_{i<j} \beta_{ij} \abs{x_i - x_j} - \abs{\sum_i x_i}\right).
\en  
What properties of the probability measure can we explicitly describe ? Clearly, any deep inspection of such a general family is extremely difficult. We will improve our chances if we restrict the edge weights to the following class. Consider $n$ positive constants $m_1, m_2, \ldots, m_n$. Let $\beta_{ij}=m_im_j$ for all pairs $i,j$. In particular the graph is complete. One can think of $m_i$ as the \textit{mass} of the $i$th particle, and hence the strength of attraction between particles $i$ and $j$ is proportional to the product of their masses. In fact, the case when all the $m_i$'s are equal to one have been dealt with in \cite{chatpal} where they were named the one-dimensional gravity model.

Now suppose $X = (X_1,\ldots,X_n)$ follows the p.d.f.
\[
C_n \exp\left(-\sum_{i<j} m_im_j \abs{x_i - x_j} - \abs{\sum_i x_i}\right),
\]
where $C_n$ is the normalizing constant. Let $M = \sum_{i=1}^n m_i$ be the total mass of the system. For each $i$, let $Y_i = X_{(i)}$ and $\Pi(i)$ be the (random) index $j$ such that $Y_i = X_j$. The joint p.m.f.\ of $(\Pi, Y)$ at a point $(\pi,y)$ is 
\[
C_n \exp\biggl(-\sum_{i<j} m_{\pi_i} m_{\pi_j} (y_j-y_i) - \abs{\sum_i y_i}\biggr),
\]
where $\pi=(\pi_1,\ldots,\pi_n)$ is a permutation of $\{1,\ldots,n\}$ and $y_1 < \cdots < y_n\in \mathbb{R}$. 
Now let $\Delta_i = Y_{i+1}-Y_i$, $i=1,\ldots,n-1$. For each $i$, and each $\pi\in S_n$, let 
\[
F_i(\pi) = \frac{\sum_{j=1}^i m_{\pi_j}}{M}.
\]
Then $M^{-1}\sum_{j=i+1}^n m_{\pi_j} = 1- F_i(\pi)$, and $F_n(\pi) \equiv 1$. The joint density of $(\Pi, \Delta, Y_1)$ at a point $(\pi, \delta, y_1)$ is
\begin{align*}
&C_n \exp \left(- \sum_{i<j} m_{\pi_i} m_{\pi_j} \biggl(\sum_{k=i}^{j-1} \delta_k  \biggr) - \abs{\sum_{i=2}^n\sum_{j=1}^{i-1} \delta_j + n y_1}\right)\\
&= C_n \exp\left(-M^2\sum_{i=1}^{n-1} F_i(\pi)(1-F_i(\pi)) \delta_i - \abs{\sum_{i=2}^n\sum_{j=1}^{i-1} \delta_j + n y_1}\right).
\end{align*}
Now we can easily integrate out $Y_1$ to get the joint density of $(\Pi, \Delta)$:
\[
C \exp\biggl(-M^2\sum_{i=1}^{n-1} F_i(\pi)(1-F_i(\pi)) \delta_i\biggr),
\]
where $C$ is now a different constant. Thus, conditional on $\Pi = \pi$, $Y_1,\ldots,Y_n$ are independent, with $Y_i \sim \mathrm{Exp}(M^2F_i(\pi)(1-F_i(\pi)))$. It is easy to see from this observation that the marginal p.m.f.\ of $\Pi$ must be
\eq\label{pmfpi}
C(m) \prod_{i=1}^{n-1} \frac{1}{F_i(\pi)(1-F_i(\pi))},
\en
where $\pi = (\pi_1,\ldots,\pi_n)$ is any permutation of $\{1,\ldots,n\}$ and $C(m)$ is the normalizing constant that depends on the values of $m_1,\ldots, m_n$.

If $m_1 = m_2 =\cdots = m_n$, then this is the uniform distribution on $S_n$. Otherwise, it is a non-uniform distribution on the set of permutations. Thus the cost we pay for losing the regularity of group actions is that the spacings between the order statistics are only conditionally independent Exponentials, conditioned on the random permutation generated by the ranks.

It is very difficult to see what sort of distributions on the space of permutations the probability mass function \eqref{pmfpi} induces. Clearly the p.m.f.\ is large when $F_i(\pi)$ is close to zero or one for most values of $i$. The intuition from gravity predicts that heavier particles should be close and should avoid being too high or too low in rank.

We now show this to be true in a particularly simple case when there is a single distinguished particle. Suppose that $m_1=\alpha$ and $m_2=m_3=\ldots=m_n=1$. We consider the joint distribution as before
\[
\frac{d\mu}{dx} = C_n \exp\left(-\sum_{i < j}m_i m_j\vert x_i - x_j\vert + \abs{ \sum x_i } \right). 
\]
We are interested in the derived joint distribution of the ranks of each particle given by the general expression in the previous section. 

Let us compute the distribution of the rank of the first particle which is distinguished from the others due to a different mass. 
\begin{equation}\label{beta}
\begin{split}
P(\text{rank of the first particle} = j)&= P(\Pi(j)=1)\\
&= \sum_{\sigma: \sigma(j)=1} C \prod_{i=1}^{n-1}\frac{1}{F_i(\sigma)(1-F_i(\sigma))},
\end{split}
\end{equation}
where $C$ is a constant depending on $n$ and $\alpha$. In the following text, we will freely use $C$ as the normalizing constant keeping in mind that the constants might be different from one another although they only depend on $n$ and $\alpha$.

Now, there are $(n-1)!$ many permutations $\sigma$ such that $\sigma(j)=1$. For any of them
\[
\begin{split}
F_i(\sigma) = 
\begin{cases}
\frac{i}{\alpha + n-1} & \text{if}\; i < j,\\
\frac{\alpha + i-1}{\alpha + n-1} & \text{otherwise}.
\end{cases}
\end{split}
\]
And thus we can rewrite \eqref{beta} as
\begin{equation}\label{finitebeta}
\begin{split}
P(\Pi(j)=1) &= C \frac{1}{\prod_{i=1}^{j-1}\frac{i}{\alpha + n - 1}\left(1- \frac{i}{\alpha+ n -1}  \right)\prod_{i=j}^{n-1} \frac{\alpha + i - 1}{\alpha + n -1}\left( 1 - \frac{\alpha + i - 1}{\alpha + n - 1} \right) }\\
&= C \frac{1}{(j-1)! (n-j)! \prod_{i=1}^{j-1}(\alpha + n - i -1) \prod_{i=j}^{n-1}(\alpha + i -1)}\\
&= C \comb{n-1}{j-1}\frac{\alpha(\alpha+1)\ldots(\alpha+n-j-1)\; \alpha(\alpha+1)\ldots(\alpha + j-2) }{\alpha(\alpha+1)\ldots(\alpha + n -2)\;\alpha(\alpha+1)\ldots(\alpha + n - 2)}\\
&= C \comb{n-1}{j-1}\alpha(\alpha+1)\ldots(\alpha+n-j-1)\; \alpha(\alpha+1)\ldots(\alpha + j-2).
\end{split}
\end{equation}
We can immediately infer from the previous expression the following fact:
\[
\begin{split}
\frac{P(\text{rank of the first particle}=j+1)}{P(\text{rank of the first particle}=j)} & = \frac{P(\Pi(j+1)=1)}{P(\Pi(j)=1)}\\
& = \frac{(n-j)(\alpha + j -1)}{j(\alpha + n - j -1)}.
\end{split}
\]
In other words
\begin{eqnarray*}
P(\text{rank of the first particle}=j+1) &>& P(\text{rank of the first particle}=j)\\
\text{iff}\quad (n-j)(\alpha + j -1) & >& j(\alpha + n - j -1),\\
\text{iff}\quad n(\alpha - 1) + (n - \alpha + 1)j - j^2 &>& (\alpha + n -1)j - j^2,\\ 
\text{iff}\quad 2(\alpha - 1)j &<& n(\alpha - 1).
\end{eqnarray*}
Thus, if $\alpha$ is more than $1$, the probability of the rank being $j$ increases till $j=\lceil n/2\rceil$, and then strictly decreases. Clearly, the most likely position for the heavier particle is going to be the median. On the other hand, if $\alpha < 1$, just the opposite happens, and we are likely to see the lighter particle either at the top or trailing behind.

The probability computed in \eqref{finitebeta}, although seemingly unfriendly, is a very familiar object. Consider a Polya's urn scheme which has $\alpha$ red balls and $\alpha$ black balls. We play a game where at each step we pick a ball at random and replace it in the urn with a ball of the same color. It is well known (see Feller \cite{feller}) that if we play this game for $n-1$ steps the probability we pick exactly $j-1$ red balls is given by
\[
\comb{n-1}{j-1}\frac{\alpha(\alpha+1)\ldots(\alpha+n-j-1)\; \alpha(\alpha+1)\ldots(\alpha + j-2)}{2\alpha (2\alpha +1)(2\alpha +2)\ldots (2\alpha + n-2)}.
\]
If we compare the previous expression with \eqref{finitebeta}, the differences are merely in the expression of the normalizing constants. Thus, if $\sigma(1)=\Pi^{-1}(1)$ is the rank of the first particle, it is clear that $\sigma(1)-1$ has the same distribution as the number of red balls picked in a Polya's urn scheme run for $n-1$ steps.

\begin{prop}\label{massbeta} For any $\alpha > 0$, consider the SDE \eqref{sdegraph} with a distribution of mass such that the mass of the first particle being $\alpha$ and the rest being of mass $1$. 
Then the sequence of random variables $\sigma(1)/n$ converges weakly to the Beta$(2\alpha, 2\alpha)$.
\end{prop}

\begin{proof}
The proof follows from known results about Polya's urn. The factor of $2$ is due to the fact that we had earlier absorbed the $2$ in \eqref{invgraph}. 
\end{proof}

\bibliographystyle{amsalpha}

\end{document}